\documentclass[twoside,a4paper,11pt,limits]{amsart}

\usepackage{amsmath}
\usepackage{amsthm}
\usepackage{amssymb}
\usepackage{amsfonts}
\usepackage{amsxtra}
\usepackage{comment}

\usepackage{enumitem}
\setenumerate{label={\rm (\alph{*})}}

\usepackage[frak,operator]{paper_diening}
\usepackage{graphicx}

\newtheorem{theorem}{Theorem}[section]
\newtheorem{lemma}[theorem]{Lemma}
\newtheorem{corollary}[theorem]{Corollary}
\newtheorem{proposition}[theorem]{Proposition}
\newtheorem{remark}[theorem]{Remark}

\numberwithin{equation}{section}

\newcommand{\meantmp}[2]{#1\langle{#2}#1\rangle}
\newcommand{\mean}[1]{\meantmp{}{#1}}

\newcommand{\Rn}{{\setR^n}}
\newcommand{\RN}{{\setR^N}}
\newcommand{\RNn}{\setR^{N\times n}}

\newcommand{\px}{{p(\cdot)}}
\newcommand{\pdx}{{p'(\cdot)}}

\newcommand{\pmi}{{p_Q^-}}
\newcommand{\ppl}{{p_Q^+}}

\newcommand{\gh}{\big(\abs{G}^\px+h\big)}
\newcommand{\dint}{\dashint}

\newcommand{\PP}{\mathcal{P}}
\newcommand{\PPln}{\mathcal{P}^{{\log}}}
\newcommand{\PPvln}{\mathcal{P}^{{\log}}_{{\rm van}}}

\newcommand{\Ms}{M^*_\Omega}
\newcommand{\Mss}{M^*_{m_0,\Omega}}

\newcommand{\pre}{{\rm pre}}
\newcommand{\cent}{{\rm center}}

\newcommand{\Uel}{{\ensuremath{U^{\kappa,\epsilon}_\lambda}}}

\begin{document}

\title{Global gradient estimates for the $p(\cdot)$-Laplacian}

\author{L.~Diening and S.~Schwarzacher}

\address{Theresienstr. 39, D-80333 Munich, Germany}

\email{lars@diening.de}
\email{schwarz@math.lmu.de}

\begin{abstract}
  We consider \Calderon{}-Zygmund type estimates for the
  non-homogeneous $p(\cdot)$-Laplacian system
  \begin{align*}
    -\divergence(\abs{D u}^{p(\cdot)-2} Du) &=
    -\divergence(\abs{G}^{p(\cdot)-2} G),
  \end{align*}
  where $p$ is a variable exponent. We show that
  $\abs{G}^{p(\cdot)} \in L^q(\Rn)$ implies $\abs{D u}^{p(\cdot)}
  \in L^q(\Rn)$ for any $q \geq 1$.  We also prove local estimates
  independent of the size of the domain and introduce new techniques to variable analysis. The paper is an extension of
  the local estimates of Acerbi-Mingione~\cite{AceM05}.
\end{abstract}

\maketitle

\section{Introduction}
\label{sec:intro}
\noindent
In recent years there has been an extensive interest in the field of
variable exponent spaces $L^\px$. Different from the classical
Lebesgue spaces $L^p$, the exponent is not a constant but a function
$p=p(x)$. 

The increasing interest was motivated by the model for
electrorheological fluids~\cite{RajR96,Ruz00}. Those are smart
materials whose viscosity depends on the applied electric field.
 This is modeled via a dependence of the viscosity on a variable exponent. 
Electrorheological fluids can for example be used in the construction of
clutches and shock absorbers.

Further applications of the variable exponent spaces can be found in
the area of image reconstruction. Here, the change of the exponent is
used to model different smoothing properties according to the edge
detector. This can be seen as a hybrid model of standard diffusion and
the TV-model introduced by~\cite{CheLevRao06}.

A model problem for image reconstruction as well as a starting point for the study of
 electrorheological fluids
 is the $\px$-Laplacian system
We consider local, weak solutions $u\in W^{1,\px}(\Omega)$ of the
non-homogeneous $p(\cdot)$-Laplacian system
\begin{align}
  \label{eq:difeq}
  -\divergence (\abs{Du}^{\px-2} Du)=-\divergence (\abs{G}^{\px-2}G).
\end{align}
where $\Omega \subset \Rn$ is an open set and $u \,:\, \Omega \to
\RN$. Note that the specific form of the right hand side is no
restriction, but allows an easier formulation of our results. 

Our main result is that $L^q$ integrability of $\abs{G}^\px$ implies
$L^q$ integrability of $\abs{Du}^\px$. We present local and global
versions of this result, see Subsection~\ref{ssub:main}. For the
exponent we assume the vanishing $\log$-H{\"o}lder continuity introduced
in~\cite{AceM01} (see \eqref{ineq:log_cont} for the definition). Our result is an extension of the results of Acerbi
and Mingione in~\cite{AceM05}, where the authors prove the local
version of the higher integrability.  The main difference
to~\cite{AceM05} is that our estimates have a controllable dependence
of the size of the ball, where higher integrability is considered.
This allows to extend the result to the whole spaces as well as to countable families of balls. The proof of these estimates
require a finer analysis of the underlying $\px$-structure. This refinement simplifies the proof significantly.

%
%

Higher integrability of the non-linear $p$-Laplace (which corresponds
to a constant exponent~$p$) was introduced in~\cite{Iwa83}. The
principle is known under the name \textit{Non-linear \Calderon-Zygmund Theory}.
In the limiting case $q=\infty$, the space $L^\infty$ has to be
replaced by~$\setBMO$ as in the linear \Calderon-Zygmund theory.
Corresponding $\setBMO$ results for the $p$-Laplace system has been
shown in~\cite{DiBMan93,DieKapSch11}. In~\cite{DieKap12,DieKapSch13} a nonlinear \Calderon-Zygmund theory was developed for the (constant) p-Stokes equation, which have some implications for the p-Navier-Stokes system. Furthermore higher integrability for small exponents and the $\px$-Stokes system was shown in~\cite{Zhi97} and \cite[Chapter 7]{DiePhd}. One future aim would be to combine these results with the variable exponent technique presented in this work to gain a nonlinear \Calderon-Zygmund theory for electrorheological fluids, that includes large exponents and BMO estimates. 
%
%

This paper is formulated only in terms of the $\px$-Laplacian in order
to simplify the notations. However, it is possible to work in a more
general setting and to consider the equation
\begin{align*}
  -\divergence(A(\cdot,Du)) = -\divergence(A(\cdot,G)).
\end{align*}
Our estimates in Section~\ref{sec:gehring} are only based on
the following two estimates:
\begin{align*}
  \abs{A(x,z)} &\leq c_1\, \abs{z}^{p(x)-1} +h_1(x),
  \\
  \abs{A(x,z)\cdot z} &\geq c_2\, \abs{z}^{p(x)}-h_2(x)
\end{align*}
for all $x \in \Omega$ and all $z \in \RNn$ and $h_1,h_2\in L^1(\Omega)\cap L^\infty(\Omega)$ . No more additional assumptions on $A$ are needed!

For Section~\ref{sec:higher-integrability} our estimates additionally need
\begin{align}
  \label{eq:diffAlog}
  \abs{A(x,z)-A(y,z)}&\leq c_3\,\abs{p(x)-p(y)}\,\bigabs{\log\abs{z}}
  \big(\abs{z}^{p(x)-1}+\abs{z}^{p(y)-1}\big),
  \\
  \label{eq:diffAV}
  \abs{z}^{p(x)} &\leq c_4 \abs{\xi}^{p(x)} + c_4 \big(
  A(x,z)-A(x,\xi)\big) \cdot(z-\xi) \big)
\end{align}
for all $x \in \Omega$ and $z,\xi \in \Rn$. However, we use some estimates (for our homogeneous comparison solution) which requires more assumptions on $A$, but not on $\px$. (See Theorem \ref{thm:hp2} and Theorem \ref{thm:hp1}). The necessary assumption for these theorems can be found in the given references.
Certainly all estimates mentioned above are valid in case of the $\px$-Laplacian;
i.e. for $A(x,z) := \abs{z}^{\px-2} z$.  Note that the same technique
allows to treat the cases $A(x,z) = (\gamma + \abs{z})^{\px-2} z$ or
$A(x,z) = (\gamma^2 + \abs{z}^2)^{\frac{\px-2}{2}} z$ for some $\gamma
\geq 0$.

The structure of the paper is as follows. In Section~\ref{sec:notation}
we introduce the necessary notation. In particular, the Lebesgue
spaces with variable exponents and the (vanishing) $\log$-H{\"o}lder
continuity is introduced. In Section~\ref{sec:gehring} we show that
the solutions to~\eqref{eq:difeq} satisfy a Gehring type
estimate. This corresponds to the higher integrability $\abs{Du}^\px \in
L^q$ with~$q$ only slightly bigger than one. The proof goes by
standard arguments via a Caccioppoli estimate and a reverse H{\"o}lder's
inequality.
 In
Section~\ref{sec:higher-integrability} we prove the main results on
higher integrability for large exponents. The arguments uses 
redistributional estimates (good-$\lambda$ estimates), which are based
on comparison estimates.

%

\section{Notation and Structure}
\label{sec:notation}
\noindent
By $c$ we denote a generic constant, whose value may change between
appearances even within a single line.  By $f\sim g$ we mean that
there exists $c$ such that $\frac1c f \le g \le c\,f$.

For a measurable set $E \subset \Rn$ let $\abs{E}$ be the Lebesgue
measure of~$E$ and $\chi_E$ its characteristic function.  For an open
set $\Omega \subset \Rn$ let $L^0(\Omega)$ denote the set of
measurable functions $f\,:\, \Omega \to \setR$ and let $L^1_{\loc}$
denote the set of locally integrable functions (integrable on compact
subsets). For $0 < \abs{E} < \infty$ and $f \in L^1(E)$ we define the
mean value of $f$ over $E$ by
\begin{align*}
  \mean{f}_E:=\dint_{E}f \,dx:=\frac{1}{\abs{E}}\int_E f \,dx.
\end{align*}
By $L^{s,\infty}(\Rn) := \set{f \in L^0(\Rn)\,:\, \norm{f}_{s,\infty}<
  \infty}$ with $s \in [1,\infty)$ and 
\begin{align*}
  \norm{f}_{s, \infty} := \sup_{\lambda>0} \norm{\lambda \,
    \chi_{\set{\abs{f}>\lambda}}}_s
  = \sup_{\lambda>0} \lambda
  \abs{\set{\abs{f}>\lambda}}^{ \frac 1s}
\end{align*}
we denote the Marcinkiewicz spaces.

Let us introduce the spaces of variable exponents~$L^\px$.  We use the
notation of the recent book~\cite{DieHHR11}.  We define $\PP(\Omega)$
to consist of all $p \in L^0(\Omega)$ with $p\,:\, \Omega \to
[1,\infty]$ (called variable exponents).  For $p \in\PP(\Omega)$ we
define $p^-_\Omega := \essinf_\Omega p$ and $p^+_\Omega :=
\esssup_\Omega p$.  For non-localized results we omit the
index~$\Omega$ of $p^+_\Omega$ and $p^-_\Omega$.  Note that the higher
integrability results in this article are restricted to the case $1 <
p^- \leq p^+ < \infty$.

For $p \in \PP(\Omega)$ with $p^+ < \infty$ the generalized Lebesgue
space $L^\px(\Omega)$ is defined as
\begin{align*}
  L^\px(\Omega):=\bigset{f \in L^0(\Omega) \,:\,
    \norm{f}_{L^\px(\Omega)} < \infty},
\end{align*}
where
\begin{align*}
  \norm{f}_{\px,\Omega}:=\norm{f}_{L^\px(\Omega)}:=
  \inf\Biggset{\lambda>0\,:\,\int_{\Rn}\biggabs{
    \frac{f(x)}{\lambda}}^{p(x)} \,dx \leq 1}.
\end{align*}
The generalized Sobolev space $W^{1,\px}(\Omega)$ consists of those
$L^1_{\loc}(\Omega)$-functions whose norm
\begin{align*}
  \norm{f}_{W^{1,\px}(\Omega)}= \norm{f}_{L^\px(\Omega)}
  +\norm{Df}_{L^\px(\Omega)};
\end{align*}
is finite, where $Df$ is the distributional derivative of $f$.

If $p$ is constant, then $L^\px$ and $W^{1,\px}$ coincide with the
classical Lebesgue and Sobolev spaces. The spaces $L^\px$ where
introduced by~\cite{Orl31}. Many properties of $L^\px$ and $W^{1,\px}$
can be found in~\cite{KovR91,FanZ01} and the book~\cite{DieHHR11}.

We say that a function $\alpha\colon \Omega \to \setR$ is {\em
  $\log$-H{\"o}lder continuous} on $\Omega$ if there exists a constant $c
\geq 0$ and $\alpha_\infty
\in \setR$ such that
\begin{align*}
  \abs{\alpha(x)-\alpha(y)} &\leq \frac{c}{\log
    (e+1/\abs{x-y})} &&\text{and}& \abs{\alpha(x) - \alpha_\infty}
  &\leq \frac{c}{\log(e + \abs{x})}
\end{align*}
for all $x,y\in \Omega$. The first condition describes the so called
local $\log$-H{\"o}lder continuity and the second the decay
condition.  The smallest such constant~$c$ is the $\log$-H{\"o}lder
constant of~$\alpha$. The decay condition is always satisfied if
$\Omega$ is bounded. We define $\PPln(\Omega)$ to consist of those
exponents $p\in \PP(\Omega)$ for which $\frac{1}{p} \,:\, \Omega \to
[0,1]$ is $\log$-H{\"o}lder continuous on $\Omega$. If $p \in \PP(\Omega)$
is bounded, then $p \in \PPln(\Omega)$ is equivalent to the
$\log$-H{\"o}lder continuity of $p$. However, working with $\frac 1p$
gives better control of the constants especially in the context of
averages and maximal functions. Therefore, we define $c_{\log}(p)$ as
the $\log$-H{\"o}lder constant of~$1/p$. Expressed in~$p$ we have for all
$x,y \in \Omega$
\begin{align}
  \label{eq:plogH}
  \abs{p(x) - p(y)} \leq \frac{(p^+)^2 c_{\log}(p)}{\log
    (e+1/\abs{x-y})} &&\text{and }& \abs{p(x) - p_\infty} &\leq
  \frac{(p^+)^2 c_{\log}(p)}{\log(e + \abs{x})}.
\end{align}

In this work cubes are always parallel to the axes and are usually
called~$Q$. We write $\ell(Q)$ for the side length of~$Q$ and
$\cent(Q)$ for the center of~$Q$.  By $\gamma Q$ with $\gamma>0$ we
mean the cube scaled by the factor~$\gamma$ with the same center
as~$Q$.

If $p \in \PPln(\Omega)$ with $p^- > 1$, then the Hardy-Littlewood
maximal operator~$M$
\begin{align*}
  (Mf)(x) := \sup_{x \ni Q} \dashint_Q \abs{f(y)}\,dy,
\end{align*}
is bounded on $L^\px(\Rn)$, where the supremum is taken over all cubes
(with sides parallel to the axes) containing~$x$. The operator norm
of~$M$ depends only on~$c_{\log}(p)$ and $p^-$.  This result goes back
to~\cite{Die04,CruFN03}. The most advanced form of this result can be
found in~\cite[Theorem 4.3.8]{DieHHR11}.

The boundedness of $M$ has many interesting consequences like Sobolev
embeddings and the boundedness of singular integrals, see
\cite{CruFMP06,DieHHR11}.

In the case of higher-integrability of the $\px$-Laplacian system a
slightly stronger condition is needed. It's local version has been introduced by
Acerbi and Mingione in~\cite{AceM01}. It's natural decay counterpart has been introduced in \cite{Schw10}. We say that a function
$\alpha\colon \Omega \to \setR$ is {\em vanishing $\log$-H{\"o}lder
  continuous} on $\Omega$ if there exists $\alpha_\infty$ such that
for every $\epsilon>0$ there exists $r,R>0$ such that
\begin{alignat}{2}
   \label{ineq:log_cont}
  \abs{\alpha(x)-\alpha(y)} &\leq \frac{\epsilon}{\log
    (e+1/\abs{x-y})}
\end{alignat}
for all $x,y$ with $\abs{x-y} \leq r$ and all $x,y$ with $\abs{x},
\abs{y} \geq R$ and
\begin{align*}
  \abs{\alpha(z) - \alpha_\infty} &\leq \frac{\epsilon}{\log(e +
    \abs{z})}
\end{align*}
for all $z$ with $\abs{z} \geq R$.
We say that $p \in \PPvln(\Omega)$ if $\frac{1}{p}$ is vanishing
$\log$-H{\"o}lder continuous. For bounded exponents this is equivalent to
the vanishing $\log$-H{\"o}lder continuity of~$p$ itself. 

The necessity of the extra ``vanishing'' condition is in analogy to
the situation of the higher integrability results for the
$p$-Laplacian with coefficients, see \cite{KinZho99}: here it is
necessary that the coefficients are in VMO (vanishing mean
oscillation) rather than just in $\setBMO$ (bounded mean oscillation).
As in~\cite{KinZho99} the vanishing $\log$-H{\"o}lder condition can be
replaced by smallness of the $\log$-H{\"o}lder constant~$c_{\log}$, see Remark \ref{rem:pVMO}. 

\section{A Gehring type estimate}
\label{sec:gehring}
\noindent
In this section we show that local solutions of the $\px$-Laplacian
system satisfy a Caccioppoli estimate. The next step is a reverse
H{\"o}lder estimate. We present the tools from the Lebesgue and Sobolev
spaces of variable exponents, which are necessary for this step. In
the end we apply the classical Gehring Lemma to the quantity
$\abs{Du}^\px$ to derive higher integrability for small
exponents.

Let us begin with the Caccioppoli estimate.
\begin{lemma}[Caccioppoli estimate]
  \label{lem:cacc}
  Let $Q\subset\Rn$ be a cube (or ball) with length~$R$ and $p\in
  \PP(2Q)$ with $p^+<\infty$. Then the local weak solution $u\in
  W^{1,\px}(2Q)$ of \eqref{eq:difeq} satisfies
  \begin{align*}
    \int_{{Q}}\abs{Du}^{\px} \,dx \leq c \left(
    \int_{2{Q}}\biggabs{\frac{u-\mean{u}_{2Q}}{R}}^{\px} \,dx +
    \int_{2{Q}}\abs{G}^{\px} \,dx\right).
  \end{align*}
  The constant depends only on $p^+$.
\end{lemma}
\begin{proof}
  The Caccioppoli estimate is proved straight forward by using the
  test function $\eta^k(u-\mean{u}_{2Q})$; here $\chi_Q\leq\eta\leq
  \chi_{2Q}$ is the usual cut off function with $\abs{D\eta}\leq\frac{c}{R}$
  and $k>p^+$ a fixed integer. 
  Therefore we get by \eqref{eq:difeq}
  \begin{align}
    \label{eq:1}
    \begin{aligned}
      \int\eta^k\abs{Du}^{\px} \,dx &\leq k\int\abs{G}^{\px-1} \abs{D \eta}
      \abs{u-\mean{u}_{2Q}}\, \eta^{k-1} \,dx
      \\
      &\quad +\,
      \int\eta^k\abs{G}^{\px-1} \abs{Du} \,dx
      \\
      &\quad + k \int \abs{Du}^{\px-1} \eta^{k-1}
      \abs{D\eta}\, \abs{u-\mean{u}_{2Q}} \,dx
      \\
      &=:(I)+(II)+(III).
    \end{aligned}
  \end{align}
  Now, $\abs{D\eta} \leq \frac cR$ and Young's inequality imply
  \begin{align*}
    (I) &\leq c\int_{2Q}\abs{G}^{\px} \,dx +
    c\int_{2Q}\biggabs{\frac{u-\mean{u}_{2Q}}{R}}^{\px} \,dx,
    \\
    (II) &\leq c_{\delta}\int_{2Q}\abs{G}^{p(\cdot)} \,dx + \delta
    \int \eta^{k\px}\abs{Du}^{p(\cdot)} \,dx,
    \\
    (III) &\leq \delta\int\eta^{(k-1)\pdx}\abs{Du}^{p(\cdot)} \,dx +
    c_\delta\int_{2Q}\biggabs{\frac{u-\left\langle u\right\rangle
        _{2Q}}{R}}^{\px} \,dx.
  \end{align*}
  Since $(k-1)\pdx \geq k$, we can absorb the terms with the
  factor~$\delta$ on the left hand side of~\eqref{eq:1} to get the claim.
\end{proof}
To deduce the reverse H{\"o}lder estimate from our Caccioppoli estimate,
we need a Sobolev-\Poincare{} inequality for variable exponents. The
proof is based on a Jensen type inequality, known as the \textit{key estimate for variable exponents}.
It first appeared in a simpler form in~\cite{Die04} and was later
improved in~\cite[Theorem~4.2.4]{DieHHR11}. We will use a further
improvement \cite[Theorem~1]{DieSch13key}, a further
improvement which allows to apply the key estimate to a larger
class of functions. The idea of this refinement goes back to
Schwarzacher~\cite{Schw10}.
\begin{lemma}[$\px$-Jensen's inequality]
  \label{lem:pjensen}
  Let $Q \subset \Rn$ be a cube (or ball), $p \in \PPln(Q)$ with
  $p^+<\infty$, $m>n$, $\beta\geq 0$ and $K_1 \geq 1$. Then
  \begin{align*}
    \bigg( \dint_{Q} \abs{f}dy\bigg)^{p(x)}\leq c\dint_{{Q}}
    \abs{f}^{p(y)}dy+c\,
    (e+\abs{x})^{-m}+c\,\dint_{Q}(e+\abs{y})^{-m}\,dy
  \end{align*}
  for all $x\in Q$ and all $f\in L^{p(\cdot)}(Q)$ satisfying
  \begin{align*}
    \dashint_Q \abs{f}\,dy \leq K_1\,\max \set{1, \abs{Q}^{-\beta}},
  \end{align*}
  where $c$ depends only on $c_{\log}(p), m, n, \beta, K_1,p^+$.
\end{lemma}
%
\begin{remark}
  Let us point out that Lemma~\ref{lem:pjensen} is valid for all
  functions $f \in L^1+L^\infty$, since
  \begin{align*}
    \dashint_{Q} \abs{f}\,dx &\leq 2\,\norm{f}_{L^1+L^\infty} \max
    \set{1, \abs{Q}^{-1}}.
  \end{align*}
  Since $L^\px+L^\infty \embedding L^1+L^\infty$,
  Lemma~\ref{lem:pjensen} is a stronger version of Theorem~4.2.4
  of~\cite{DieHHR11}, which proves the same result under the condition
  $f \in L^\px+L^\infty$.
%
\end{remark}
\begin{remark}
  \label{rem:dwK}
  Whenever $x \in Q$ satisfies $\abs{x}= \sup_{y\in Q}\abs{y}$, then
  by the geometry of $(e+\abs{\,\cdot\,})^{-m}$, we have
  \begin{align*}
    (e+\abs{x})^{-m}&\leq \dint_{Q}(e+\abs{y})^{-m}\,dy.
  \end{align*}
  In this case we can remove one term in the estimate of
  Lemma~\ref{lem:pjensen}. 
\end{remark}

Based on the new key estimate it was possible to prove an
refined version of the Sobolev
\Poincare{}inequality. See~\cite[Proposition~8.2.11]{DieHHR11} and \cite[Corollary~3]{DieSch13key} for a proof. 
\begin{proposition}[Sobolev Poincar{\'e}]
  \label{prop:sp}
  Let $Q \subset \Rn$ be a cube (or ball) with length~$R$, $p \in
  \PPln(Q)$ with $p^+<\infty$, and $f\in W^{1,\px}(Q)$ with
  $\norm{Df}_{L^{1}+L^{\infty}}\leq K_2$. For $s \in
  [1,\min\{\frac{n}{n-1}, \pmi\})$ and $m>n$ there exists a constant
  $c$ depending on $s,\pmi, c_{\log}(p),m,n, K_2$ for which
  \begin{align*}
    \dint_{Q}\left( \frac{\abs{f-\mean{f}_{Q_R}}}{R}\right) ^{\px} dx
    &\leq c\left( \dint_{Q}\abs{Df}^{\frac{p(\cdot)}{s}} \,dx\right)^s
    + c\dint_{Q}h \,dx,
  \end{align*}
  with
  \begin{align}
    \label{eq:h}
    h(x):=(e+\abs{x})^{-m}.
  \end{align}
\end{proposition}
It is possible to replace the condition
$\norm{Df}_{L^1+L^\infty}$ by the condition of Lemma~\ref{lem:pjensen}
applied to~$Df$.

Now, the reverse H{\"o}lder estimate follows directly by
Lemma~\ref{lem:cacc} and Proposition~\ref{prop:sp}.
\begin{lemma}[Reverse H{\"o}lder estimate]
  \label{lem:rh}
  Let $Q\subset\Rn$ be a cube (or ball) and $p\in \PPln(2Q)$. For a
  local weak solution $u\in W^{1,\px}(2Q)$ of \eqref{eq:difeq} we have
  \begin{align*}
    \dint_{{Q}}\abs{Du}^{\px} \,dx \leq c\Bigg( \dint_{2{Q}}\abs{Du}^{\px/s}
    \,dx \Bigg)^s+c\dint_{2{Q}}\abs{G}^{\px} \,dx+c\dint_{2{Q}}h \,dx,
  \end{align*}
  for all $s \in [1, \min\{p^-,\frac{n}{n-1}\})$. The constant depends
  on $\ppl,c_{\log}(p),m,n,s$ and $\norm{Du}_{L^\infty+L^1(2Q)}$.
\end{lemma}

Let us restate Gehring's Lemma at this
point~\cite[Section~4]{Iwa98}.
\begin{theorem}[Gehring's lemma]
  \label{thm:oldgehring} 
  Let $f\in L^s(\Omega)$ and $g\in L^{q}(\Omega)$.
  If the reverse H{\"o}lder inequality
  \begin{align*}
    \Bigg(\dint_{Q} \abs{f}^s \,dx\Bigg)^\frac{1}{s}\leq
    C_{Ge}\dint_{2Q}\abs{f} \,dx+\Bigg(\dint_{2Q}\abs{g}^s
    \,dx\Bigg)^\frac{1}{s}
  \end{align*}
  is satisfied for an $s>1$ and all $2Q\subset\Omega$, then there
  exists an $m_0>1$ depending on $c_{Ge}, s,q$ and the dimension such
  that
  \begin{align*}
    \Bigg(\dint_{Q} \abs{f}^{s\mu} \,dx\Bigg)^\frac{1}{s\mu}\leq
    c\Bigg(\dint_{2Q}\abs{f}^s
    \,dx\Bigg)^\frac{1}{s}+c\Bigg(\dint_{2Q}\abs{g}^{s\mu}
    \,dx\Bigg)^\frac{1}{s\mu},
  \end{align*}
  for all $1<\mu<m_0$, all $2Q\subset\Omega$. The constant $c$ depends
  on $n,s,C_{Ge}$.
\end{theorem}
The following corollary is a consequence of Gehring's Lemma and
Lemma~\ref{lem:rh}.
\begin{corollary}
  \label{cor:gehring}
  Let $p\in \PPln(\Omega)$, $\abs{G}^\px \in L^q(\Omega)$ with $q >
  1$. Let $u\in W^{1,\px}(\Omega)$ be a local weak solution of
  \eqref{eq:difeq}. Then
  there exists $m_0\in (1,q]$ such that for all cubes $Q$ with
  $2Q\subset\Omega$ and all $\mu \in [1,m_0]$ there holds
  \begin{align*}
    \Bigg(\dint_{{Q}}\abs{Du}^{\px \mu} \,dx\Bigg)^\frac{1}{\mu} \leq
    c\dint_{2{Q}}\abs{Du}^{\px} \,dx + c\Bigg( \dint_{2{Q}}\abs{G}^{\px\mu}
    \,dx+\dint_{2{Q}}h^\mu \,dx \Bigg)^\frac{1}{\mu},
  \end{align*}
  where $h(x) = (e+\abs{x})^{-m}$ with $m>n$. The constant~$c$
  depends on $n,c_{\log}(p),p^-$, $p^+$, $m$ and
  $\norm{Du}_{L^1+L^\infty(\Omega)}$.
\end{corollary}
\begin{remark}
  \label{rem:scaling}
  By a standard covering argument it is possible to replace the pair $(Q,2Q)$
  in Lemma~\ref{lem:cacc}, Lemma~\ref{lem:rh},
  Theorem~\ref{thm:oldgehring} and Corollary~\ref{cor:gehring} by
  $(Q,aQ)$ for any $a>1$. The constant then depends on~$a$.
\end{remark}

\section{Higher integrability}
\label{sec:higher-integrability}
\noindent
In this section we prove the higher integrability of our local
solutions. We will derive a local result, with controllable dependence on the size of the ball,  Theorem \ref{thm:local}; such that the global result follows as a corollary. For better readability we split the section into several
parts. Firstly, we recall the technique of redistributional estimates
(good~$\lambda$-estimates). Secondly, we split the corresponding level
sets into cubes. Thirdly, we define a local comparison problem of
$p$-Laplace type with constant exponent. Fourth, we derive estimates
controlling the distance of the local auxiliary problem to our
original system. This enables us in our fifth step to proof our main
result of higher integrability.

\subsection{Redistributional Estimate}
\label{ssec:red}

The higher integrability of our solutions will be achieved by
redistributional estimates also known as good-$\lambda$-estimates. Let
us briefly describe this well known technique: Assume $f$ and $g$ to
be integrable, non-negative functions.  Moreover, assume that the
following redistributional estimate holds: There exists $\kappa>1$, $\epsilon>0$ and
$\delta=\delta(\epsilon)>0$ with $\delta(\epsilon)\to 0$ for
$\epsilon\to 0$ such that for all $\lambda>0$
\begin{align}
  \label{eq:gl}
  \abs{\{\abs{f}>\kappa\lambda\}\cap\{\abs{g}\leq\epsilon\lambda\}}\leq
  \delta\abs{\{\abs{f}>\lambda\}}.
\end{align}
A direct consequence of this estimate is
\begin{align*}
  \abs{\{\abs{f}>\kappa\lambda\}}\leq \delta\abs{\{\abs{f}>\lambda\}}
  + \abs{\set{\abs{g}>\epsilon\lambda}},
\end{align*}
which basically shows that the level sets of $f$ can be controlled in
a certain sense by the ones of~$g$. (Later in our setting we will
choose $\kappa=2^{n+1} c_4$ with $c_4$ from~\eqref{eq:diffAV}).
Multiplying this estimate by $\lambda^{q-1}$ with $q \in [1,\infty)$
and integrating over $\lambda \in (0,\infty)$ gives for suitable small
$\delta$ (formally)
\begin{align*}
  \int_0^\infty \lambda^q \abs{\{\abs{f}>\lambda\}}\,d\lambda\leq
  c\,\int_0^\infty \lambda^q \abs{\{\abs{g}>\lambda\}} \,d\lambda,
\end{align*}
where $c$ depends on $\kappa$, $\epsilon$ and $q$. In other
words,
\begin{align*}
  \norm{f}_q &\leq c\, \norm{g}_q. 
\end{align*}
We will apply this argument to the functions $f=\Ms(\abs{Du}^\px)$ and
$g=\Mss(\abs{G}^\px+h)$, where $\Ms$ and $\Mss$ are localized, dyadic
maximal operators which we will introduce below and $h(x) := (e+
\abs{x})^{-2n}$.


\subsection{Maximal operators and coverings}
\label{ssec:max}

Let us introduce the localized maximal operators.  By $\Delta$ we
denote the standard set of (open) dyadic cubes $(2^k a) + (0,2^k)^n$
with $k \in \setZ$ and $a \in \setZ^n$.  Now, take an arbitrary, open
cube $Q' \subset \Rn$ and let $T\,:\,\Rn\to \Rn$ be a linear mapping,
that maps $(0,1)^n$ onto $Q'$. Then cubes $\set{T(Q)\,:\, Q \in
  \Delta}$ are called the $Q'$-dyadic cubes.  By $\Delta_{Q'}$ we
denote the $Q'$-dyadic sub-cubes of~$Q'$, i.e.  $\Delta_{Q'} := \set{
  T(Q) \,:\, T(Q) \subset Q', Q \in \Delta}$. Note that two dyadic
cubes from $\Delta_{Q'}$ are either disjoint or one is a subset of the
other. The predecessor of a $Q'$-dyadic cube $Q$ is the unique
$Q'$-dyadic cube $Q^\pre$, which contains~$Q$ and has double the
diameter of~$Q$.

For $s \in [1,\infty)$ we define for all~$x \in \Rn$
\begin{align*}
  (M^*_{Q',s} f)(x)&:=\sup_{Q \in \Delta_{Q'}\,:\, x \in Q}
  \bigg(\dint_{2Q}\abs{f}^s\,dz \bigg)^{\frac 1s},
  \\
  (M^*_{Q'} f)(x)&:= (M^*_{Q',1} f)(x),
\end{align*}
where the supremum is taken over all $Q'$-dyadic sub-cubes of~$Q'$
which contain~$x$ in its closure. In particular, $M^*_{Q',s}f$ is zero
outside of $Q'$ and depends only on the values of~$f$
on~$2Q'$.  It is well known that $M_{Q',s}$ is bounded from $L^q(2Q')$
to $L^q(\Rn)$ for $q>s$ and from $L^s(\Rn)$ to the Marcinkiewicz space
$L^{s,\infty}(\Rn)$.
 
From now on we fix an open cube~$\Omega \subset \Rn$ such that $u$ is
a weak local solution of~\eqref{eq:difeq} on $2 \Omega$. Our goal is
to prove higher integrability of $\abs{Du}^{\px}$ on $\Omega$.

We define our level sets
%
\begin{align}
  \label{eq:OlUl}
  \begin{aligned}
    \mathcal{O}_\lambda&:=\{\Ms(\abs{Du}^\px)>\lambda\} \subset
    \Omega,
    \\
    \Uel&:=\{\Ms(\abs{Du}^\px)>\kappa \lambda,
    \,\Mss(\gh)\leq\epsilon\lambda\} \subset \Omega.
  \end{aligned}
\end{align} 
Our goal is to show
\begin{align}
  \label{eq:glb}
  \abs{\Uel} \leq \delta\, \abs{\mathcal{O}_\lambda}
\end{align}
with $\delta=\delta(\epsilon)\to 0$ for $\epsilon\to 0$. This together
with the arguments similar to the ones in the subsection will give the
desired result of higher integrability.

Since we are mainly interested in a result of local higher
integrability, it suffices to consider~\eqref{eq:glb} for large values
of~$\lambda$. In particular, we define
\begin{align}
  \label{eq:lambda0}
  \lambda_0 := \dashint_{2\Omega} \abs{Du}^\px \,dx.
\end{align}
We can assume without loss of generality that $\lambda_0>0$, since
otherwise $u$ is locally a constant.

To prove~\eqref{eq:glb} we will decompose $\mathcal{O}_\lambda$ into
suitable dyadic cubes, which we will construct now. For every $x \in
\mathcal{O}_\lambda$, there exists a largest $\Omega$-dyadic
cube~$Q_x$ with the property $\dashint_{2Q_x} \abs{Du}^\px
\,dx>\lambda$. In particular, any $Q' \in \Delta_\Omega$ with $Q'
\supsetneq Q_x$ satisfies $\dashint_{2Q'} \abs{Du}^\px \,dx\leq
\lambda$. 

The family~$\set{Q_x \,:\, x \in \mathcal{O}_\lambda}$ covers the
set~$\mathcal{O}_\lambda$. Since the dyadic cubes have a natural order
(if two dyadic cubes intersect, one of them contains the other), the
sub-family of maximal cubes still covers~$\mathcal{O}_\lambda$. We
denote this at most countable sub-family by $\set{Q_j}$. In
particular, we have
\begin{align*}
  \lambda &< \dashint_{2Q_j} \abs{Du}^\px \,dx.  
\end{align*}
Since $\lambda \geq \lambda_0$, we have
\begin{align*}
  \dashint_{2\Omega} \abs{Du}^\px \,dx = \lambda_0 \leq \lambda
  < \dashint_{2Q_j} \abs{Du}^\px \,dx \leq
  \frac{\abs{\Omega}}{\abs{Q_j}} \dashint_{2\Omega} \abs{Du}^\px \,dx.
\end{align*}
This implies $\abs{Q_j} < \abs{\Omega}$. In particular, we
know that $Q_j$ is a proper $\Omega$-dyadic sub-cube of~$\Omega$. Let
$Q_j^\pre$ denote the $\Omega$-dyadic predecessor of~$Q_j$.  Then
$Q_j^\pre \in \Delta_\Omega$. Since the $Q_j$ (former $Q_x$) were
chosen to be maximal we have $\dashint_{2Q_j^\pre} \abs{Du}^\px \,dx
\leq \lambda$. This and $3Q_j \subset 2Q_j^\pre$ implies
\begin{align}
  \label{eq:Qjlambda}
  \lambda &< \dashint_{2Q_j} \abs{Du}^\px \,dx \leq \frac{3^n}{2^n}
  \dashint_{3Q_j} \abs{Du}^\px \,dx \leq 2^{n} \dashint_{2Q_j^\pre}
  \abs{Du}^\px \,dx \leq 2^n \lambda.
\end{align}

Our goal was to prove the estimate $\abs{\Uel} \leq
\delta\, \abs{\mathcal{O}_\lambda}$.
Since the dyadic $Q_j$ cover the set $\mathcal{O}_\lambda$ it suffices
to prove
\begin{align}
  \label{eq:glb2}
  \abs{Q_j \cap \Uel} \leq \delta\, \abs{Q_j}.
\end{align}
We prove this estimate in the next section.  This estimate is obvious
if $Q_j \cap \Uel = \emptyset$. Therefore, we will
assume in the following, that
\begin{align*}
  Q_j \cap \Uel \not=\emptyset.
\end{align*}
In this case we find $x_j \in Q_j \cap \Uel$ with
$\Mss(\gh)(x_j) \leq \epsilon\lambda$, which implies
\begin{align}
  \label{eq:QjGlambda}
  \bigg( \dashint_{3Q_j} \gh^{m_0} \,dx \bigg)^{\frac{1}{m_0}} &\leq
  \bigg( \frac{3^n}{2^n} \dashint_{2Q_j^\pre} \gh^{m_0}\,dx
  \bigg)^{\frac{1}{m_0}}\leq \frac{3^n}{2^n} \epsilon \lambda.
\end{align}
By Corollary~\ref{cor:gehring}, Remark~\ref{rem:scaling},
\eqref{eq:Qjlambda} and~\eqref{eq:QjGlambda} we have
\begin{align}
  \label{eq:Qjpmlambda}
  \begin{aligned}
    \bigg(\dashint_{2Q_j} \abs{Du}^{\px m_0} \,dx
    \bigg)^{\frac{1}{m_0}} &\leq c \dashint_{3Q_j} \abs{Du}^\px \,dx
    + c\, \bigg(\dashint_{3Q_j} \!\!\!\gh^{m_0} dx \bigg)^{\frac{1}{m_0}}
    \\
    &\leq c\, \lambda.
  \end{aligned}
\end{align}


\subsection{Comparison Problem}
\label{ssub:compp}

If the right hand side $G$ of our system is locally zero and $p$ is
locally constant, then $u$ is locally a $p$-harmonic function with
all its nice regularity properties. If $G$ is non-zero but ``small'',
then $u$ is still close to a $\px$-harmonic function (where $G=0$).
This allows to transfer some of the regularity results from the
$\px$-harmonic function to~$u$. This is the well known comparison
principle. Unfortunately, $\px$-harmonic functions do not have as nice
regularity properties as $p$-harmonic functions with constant~$p$; so
it makes more sense to compare $u$ to a $p_i$-harmonic function, where
$p_j$ is a local constant approximation of~$\px$ on $Q_j$. Certainly, for a good
comparison $p$ should not vary to much.

In the following we will define our comparison system.  For every cube
$Q_j$ let $y_j$ denote a point of $2Q_j$ furthest away from the point
zero, i.e we choose $y_j \in \overline{2Q_j}$ such that $\abs{y_j} =
\sup_{x\in 2Q_j} \abs{x}$. Now, define $p_j := p(y_j)$ as an
approximation of~$p$.
%
We define our comparison system by
\begin{align}
  \label{eq:lhp}
  \begin{aligned}
    \int_{2Q_j }A_j(Dw_j)D\phi \,dx&=0\text{ for all }\phi\in
    W^{1,p_j}_0(2Q_j)
    \\
    w_j&=u\text{ on }\partial (2Q_j),
  \end{aligned}
\end{align}
where $A_j(Dw_j) := \abs{Dw_j}^{p_j-2}Dw_j$. We are looking for
solutions~$w_j$ in $W^{1,p_j}(2Q_j)$.

Note that for small cubes~$Q_j$ the choice of $p_j$ is not important
and one could take any $p(x)$ with $x \in Q_j$. For example Acerbi and
Mingione used in~\cite{AceM05} the choice $p_{Q_j}^+$. However, for
large cubes this is not a good choice. It is more reasonable to take
an exponent, which is close to the average of~$p$ over~$Q_j$ (actually
the best choice is the average defined by the reciprocal). Due to the
$\log$-H{\"o}lder continuity of~$p$ our choice of~$p_j$ has this property.

Note that it is a~priori not clear that our comparison
system~\eqref{eq:lhp} is well defined. This is due to the fact, that
the stated boundary condition on~$w_j$ requires $u \in
W^{1,p_j}(2Q_j)$. Since $p_j$ might be bigger than $\px$ at some parts
of $2Q_j$, this does not follow from $u \in W^{1,\px}(2Q_j)$. However,
it follows from Corollary~\ref{cor:gehring} that $u \in
W^{1,m_0\px}(2Q_j)$ for some $m_0 > 1$. So if $\px$ does not vary to
much on $2Q_j$ in the sense that $m_0 p^-_{2Q_j} \geq p^+_{2Q_j}$, then
$m_0p(x) \geq p_j$. This implies $u \in W^{1,p_j}(2Q_j)$ and our
comparison system~\eqref{eq:lhp} is well defined. It is now standard
that the system has a unique solution $w_j \in u + W^{1,p_j}_0(2Q_j)$.

We will later have a similar problem, when passing back from $w_j$ to
$u$, since $w_j \in W^{1,p_j}(2Q_j)$ is not enough to deduce $w_j \in
W^{1,\px}(2Q_j)$. For this, we need to control of $w_j$ in the space
$W^{1,p_{2Q_j}^+}(2Q_j)$. This is possible due to the following result of higher integrability (up to the boundary) \cite[Theorem~1.1]{KilKos94}. More precisely we use a quantitative estimate which is a consequence from (3.4) and Lemma~2.4 in this work.
\begin{theorem}
  \label{thm:hp2}
  There exists a $m_1>1$ such that for all $m_1\geq \mu\geq 1$ the following holds.  If $u\in W^{1,\mu p_j}(2Q_j)$ and $w_j$ is the
  solution of~\eqref{eq:lhp}, then there exists a constant $c$
  depending on $p_j$ such that
  \begin{align*}
    \dint_{2Q_j}\abs{Dw_j}^{\mu p_j} \,dx\leq
    c\dint_{2Q_j}\abs{Du}^{\mu p_j} \,dx.
  \end{align*}
\end{theorem}
 Let us point out, that the paper~\cite{KilKos94} is stated for equations only; however, all their arguments used for the estimate below are valid for systems as well.
For equations the last Theorem holds for all $1\leq\mu<\infty$ which was proven in~\cite[Theorem~5]{KinZho01}.

The use of this theorem requires higher integrability of
$Du$. To close this argument, we will assume in the following
that for every $Q_j$ with $Q_j\cap\Uel\neq\emptyset$ it holds
%
\begin{align}
  \label{eq:assQj}
  \sigma p^-_{2Q_j} \geq p^+_{2Q_j} \qquad \text{with }\sigma :=
  \sqrt[4]{\min\set{m_0,m_1}}>1.
\end{align}

In this situation Corollary~\ref{cor:gehring} implies $u \in W^{1,\sigma^4 \px}(2Q_j)
\embedding W^{1,\sigma^3 p_j}(2Q_j)$ and then Theorem~\ref{thm:hp2}
implies $w_j \in W^{1,\sigma^3 p_j}(2Q_j) \embedding
W^{1,\sigma^2\px}(2Q_j)$.

The above estimates of $u$ in $W^{1,\sigma^3 p_j}(2Q_j)$ and $w$ in
$W^{1,\sigma^2 \px}(2Q_j)$ depend unfortunately on the size of $Q_j$.
We derive in the following three lemmas more precise modular
estimates.
\begin{lemma}
  \label{lem:upj}
  If $Q_j \cap \Uel \not=\emptyset$ and $p$ holds \eqref{eq:assQj}, then
  \begin{align*}
    \bigg(\dint_{2Q_j}\abs{Du}^{\sigma^3p_j}
    \,dx\bigg)^\frac{1}{\sigma^3} & \leq c\,\lambda.
  \end{align*}
\end{lemma}
\begin{proof}
  We want to apply the key estimate
  (Lemma~\ref{lem:pjensen}) to the function $\abs{Du}^{\sigma^3
    p_j}$. So let us verify the requirements. By
  Corollary~\ref{cor:gehring}, Remark~\ref{rem:scaling} and Young's
  inequality we deduce
  \begin{align}
    \label{eq:2}
    \begin{aligned}
      \lefteqn{\dint_{2Q_j}\abs{Du}^{\sigma^3 p_j} \,dx \leq
        \dint_{2Q_j}\abs{Du}^{m_0 \px}\,dx + 1} \qquad &
      \\
      &\leq c\bigg(\dint_{3Q_j}\abs{Du}^\px \,dx\bigg)^{m_0} +c
      \dint_{3Q_j}\abs{G}^{\px m_0}+h^{m_0}\,dx + 1
      \\
      &\leq c\bigg(\dint_{3Q_j}\abs{Du}^\px \,dx\bigg)^{m_0} +c
      \dint_{3Q_j}\abs{G}^{\px q}+h^{m_0}\,dx + c
      \\
      &\leq c\, \max \bigset{ \abs{Q_j}^{-m_0}, 1},
    \end{aligned}  
  \end{align}
  where the last constant  depends on
  $\norm{\abs{Du}^\px}_{L^1(3Q_j)}$ and $\norm{\abs{G}^\px}_{L^q(3Q_j)}$.

  This allows to apply Lemma~\ref{lem:pjensen} to
  $\abs{Du}^{\sigma^3\, p_j}$ with exponent $\frac{\sigma
    \px}{p_j}\geq 1$ at $x=y_j$ (recall $p(y_j)=p_j$ and
  $\sigma^4=\min\set{m_0,m_1}$) and use Remark~\ref{rem:dwK}.
  \begin{align*}
    \bigg( \dint_{2Q_j}\abs{Du}^{\sigma^3 p_j}
    \,dx\bigg)^{\frac{\sigma p(y_j)}{p_j}} &\leq c
    \dint_{2Q_j}\abs{Du}^{m_0\px} \,dx+c\dint_{2Q_j} (e+\abs{x})^{-m_0
      2n} \,dx.
  \end{align*}
  Due to~\eqref{eq:Qjpmlambda} we can estimate the right-hand side by
  $\lambda^{m_0}$.
\end{proof}
\begin{lemma}
  \label{lem:wpj}
  If $Q_j \cap \Uel \not=\emptyset$, then
  \begin{align*}
    \bigg(\dint_{2Q_j}\abs{Dw_j}^{\sigma^3p_j}
    \,dx\bigg)^\frac{1}{\sigma^3} & \leq c\,\lambda.
  \end{align*}
\end{lemma}
\begin{proof}
  The lemma is an immediate consequence of Theorem~\ref{thm:hp2}
  and~Lemma~\ref{lem:upj}.
\end{proof}
\begin{lemma}
  \label{lem:dwdu}
  If $Q_j \cap \Uel \not=\emptyset$, then
  \begin{align*}
    \bigg(\dint_{2Q_j}\abs{Dw_j}^{\sigma^2\px}
    \,dx\bigg)^\frac{1}{\sigma^2} \leq c\,\lambda.
  \end{align*}
\end{lemma}
\begin{proof}
  We want to apply the key estimate 
  (Lemma~\ref{lem:pjensen}) to the function $\abs{Dw_j}^{\sigma^2 \px}$.
  So let us verify the requirements. By Young's inequality,
  Theorem~\ref{thm:hp2} and~\eqref{eq:2} we have
  \begin{align*}
    \dint_{2Q_j}\abs{Dw_j}^{\sigma^2 \px} \,dx &\leq
    \dint_{2Q_j}\abs{Dw_j}^{\sigma^3 p_j}\,dx + 1
    \\
    &\leq c\, \dint_{2Q_j}\abs{Du}^{\sigma^3 p_j}\,dx + 1
    \\
    &\leq c\, \max \bigset{ \abs{Q_j}^{-m_0}, 1}.
  \end{align*}
  This allows to apply Lemma~\ref{lem:pjensen} to
  $\abs{Dw_j}^{\sigma^2\, \px}$ with exponent $\frac{\sigma
    p_j}{\px}\geq 1$ at $x=y_j$ (recall $p(y_j)=p_j$ and
  $\sigma^4=\min\set{m_0,m_1}$) and use Remark~\ref{rem:dwK}.
  \begin{align*}
    \bigg( \dint_{2Q_j}\abs{Dw_j}^{\sigma^2 \px}
    \,dx\bigg)^{\frac{\sigma p_j}{p(y_j)}} &\leq c
    \dint_{2Q_j}\abs{Dw_j}^{\sigma^3 p_j} \,dx+c\dint_{2Q_j}
    (e+\abs{x})^{-\sigma^3 2n} \,dx.
  \end{align*}
  The first term on the right hand side is estimated by
  Lemma~\ref{lem:wpj}. The second term is controlled
  by~\eqref{eq:QjGlambda}.
\end{proof}

\subsection{Comparison Estimate}
\label{ssub:compe}

We show in this subsection that our approximate solution~$w_j$ is
indeed close to our solution~$u$. Obviously, the (small) distance from $A$ to
$A_j$ is most important for our estimates. Due to~\eqref{eq:diffAlog} we have
\begin{align}
  \label{eq:AA_j}
  \abs{A(x,z)-A_j(z)}\leq c\abs{p(x)-p_j}\,\bigabs{\log \abs{z}}
  \big(\abs{z}^{p(x)-1}+\abs{z}^{p(y)-1}\big),
\end{align}
for all $x \in 2\Omega$ and $z \in \RNn$. In particular, the distance
of $A$ and $A_j$ is strongly connected with the distance from~$p$ to
$p_j$. We begin with some auxiliary estimates.

Due to the special choice of $y_j$, namely $\abs{y_j} = \sup_{x\in
  2Q_j} \abs{x}$, and $p(y_j)=p_j$ it follows from~\eqref{eq:plogH}
that 
\begin{align}
  \label{eq:pvspj}
  \abs{p(x) - p_j} &\leq \abs{p(x) - p_\infty} + \abs{p(y_j)-
    p_\infty} \leq \frac{2\,(p^+)^2 c_{\log}(p)}{\log(e + \abs{x})}
\end{align}
for all $x \in 2Q_j$.
\begin{lemma}
  \label{lem:ppjavg}
  Let $Q \subset \Rn$ be a cube (or ball) with side length~$R$ and let $p
  \in \PPln(Q)$. Then for every $s \geq 1$ there exists a constant~$c$
  depending only on $s$ such that
  \begin{align*}
    \bigg(\dashint_{Q} \abs{\px - p_j}^s \,dx\bigg)^\frac{1}{s} \leq
    \frac{c\, (p^+)^2 c_{\log}(p)}{\log(e + \max \set{R, 1/R,\abs{{\rm
          center}(Q)}})}.
  \end{align*}

\end{lemma}
\begin{proof}
  If $\max \set{R, 1/R,\abs{\rm center(Q)}}=1/R$, then $R\leq 1$ and by
  the local estimate of~\eqref{eq:plogH}
  \begin{align*}
    \bigg(\dashint_{Q} \abs{\px - p_j}^s \,dx\bigg)^\frac{1}{s} &\leq
    p_Q^+ - p_Q^- \leq \frac{2(p^+)^2}{\log(e + 1/(\sqrt{n}R))} \leq
    \frac{c\, (p^+)^2 c_{\log}(p)}{\log(e + 1/R)},
  \end{align*}
  where the constant depends on $p^+$.

  In the following let $R \geq 1$. Then by~\eqref{eq:pvspj}
  \begin{align}
    \label{eq:4}
    \bigg(\dashint_{Q} \abs{\px - p_j}^s \,dx\bigg)^\frac{1}{s} &\leq
    (p^+)^2 c_{\log}(p) \bigg(\dashint_{Q} \frac{c}{(\log(e+\abs{x}))^s}
    \,dx\bigg)^\frac{1}{s}.
  \end{align}
  If $\sqrt{n}R \leq \frac{1}{2} \abs{{\rm center}(Q)}$, then $\max
  \set{R, 1/R,\abs{\rm center(Q)}}= \abs{{\rm center}(Q)}$ and
  $\abs{x} \geq \frac{1}{2} \abs{{\rm center}(Q)}$ for all $x \in
  Q$. Hence, 
  \begin{align*}
     \bigg(\dashint_{Q} \frac{c}{(\log(e+\abs{x}))^s}
     \,dx\bigg)^\frac{1}{s} \leq
     \frac{c}{\log(e+\frac 12 \abs{{\rm center}(Q)})} \leq 
     \frac{c}{\log(e+\abs{{\rm center}(Q)})}.
  \end{align*}
  It remains to consider the case $\sqrt{n}R \geq \frac{1}{2}
  \abs{{\rm center}(Q)}$ and $R\geq 1$.  In this situation we have
  $\max \set{R, 1/R,\abs{\rm center(Q)}} \leq 2 \sqrt{n} R$.  Let
  $Q_R(0)$ denote the cube of same size as~$Q$ but centered at zero.
  Then with $R \geq 1$
  \begin{align*}
    \lefteqn{ \bigg(\dashint_{Q} \frac{c}{(\log(e+\abs{x}))^s}
      \,dx\bigg)^\frac{1}{s} \leq \bigg(\dashint_{Q_R(0)}
      \frac{c}{(\log(e+\abs{x}))^s} \,dx\bigg)^\frac{1}{s}} \qquad &
    \\
    &= c\, \bigg( R^{-n} \int_0^R \frac{r^{n-1}}{(\log(e+r))^s} \,dr
    \bigg)^{\frac 1s},
    \\
    &\leq c\, \bigg( R^{-n} \int_0^{R^{\frac 12}}
    \frac{r^{n-1}}{(\log(e+r))^s} \,dr + R^{-n} \int_{R^{\frac 12}}^R
    \frac{r^{n-1}}{(\log(e+r))^s} \,dr \bigg)^{\frac 1s}
    \\
    &\leq c\, \bigg( R^{-n+\frac 12}\, R^{n-1} + R^{-n+1}
    \frac{c\,R^{n-1}}{(\log(e+R))^s} \bigg)^{\frac 1s}
    \\
    &\leq cR^{-\frac{1}{2}}+\frac{c}{(\log(e+R))}\leq \frac{c}{(\log(e+R))},
  \end{align*}
  where we used that $\log(e+r) \geq c\, \log (e+R)$ for $r \in
  [R^{\frac 12},R]$. This and~\eqref{eq:4} prove the remaining case.
\end{proof}
We need another technical lemma that takes care of the logarithmic
factor in~\eqref{eq:AA_j}.
\begin{lemma}
  \label{lem:lnfmean}
  Let $Q \subset \Rn$ be a cube and $s \geq 1$. Then there exists a
  constant $c$ depending only on~$s$ such that every $f \in
  L^1(Q)$ satisfies
  \begin{align*}
    \dashint_{Q} \log \bigg(e+ \frac{\abs{f}}{\dashint_Q \abs{f}\,dx} 
    \bigg)^s \,dx &\leq c.
  \end{align*}
\end{lemma}
\begin{proof}
  It suffices to proof the estimate for $f$ with
  $\dashint_Q \abs{f}\,dx = 1$.
  We estimate 
  \begin{align*}
    \dashint_Q (\log(e+\abs{f}))^s \,dx &= \frac{1}{\abs{Q}} \int_Q
    \int_0^{\abs{f}} s(\log(e+t))^{s-1} \frac{1}{e+t} \,dt \,dx
    \\
    &=\frac{1}{\abs{Q}} \int_0^\infty s(\log (e+t))^{s-1} \frac{1}{e+t}
    \abs{ Q \cap \set{\abs{f}> t}} \,dt. \qquad \qquad
  \end{align*}
  We split the domain of integration into $(0,1)$ and $(1,\infty)$ and
  use the estimate $t\,\abs{Q\cap\set{\abs{f}> t}} \leq \int_Q
  \abs{f}\,dx$ to get
  \begin{align*}
    \dashint_Q (\log(e+\abs{f})^s \,dx &\leq \int_0^1 s
    (\log(e+t))^{s-1} \frac{1}{e+t} \,dt
    \\
    &\quad+ \frac{1}{\abs{Q}} \int_1^\infty s(\log(e+t))^{s-1}
    \frac{1}{(e+t)t} \int_Q \abs{f}\,dx \,dt
    \\
    &\leq c.\qedhere
  \end{align*}
\end{proof}

Let us now turn to the closeness of $Dw_j$ and~$Du$.
\begin{proposition}
  \label{pro:comparison}
  For every $\kappa>0$ and $\delta>0$ there exists a $\delta_1>0$ and
  $\epsilon>0$, such that the following holds for every $\lambda>\lambda_0$ (defined in~\eqref{eq:lambda0}).

 If  $Q_j\cap \Uel
  \not=\emptyset$ and $c_{\log}(p|_{2Q_j})\leq\delta_1$, then 
  \begin{align*}
    \dint_{2Q_j}\big(A(\cdot,Du)-A(\cdot,Dw_j))\cdot\big(Du-Dw_j\big)
    \,dx\leq \delta \lambda.
  \end{align*}
  Here $\delta_1,\epsilon$ depends on $\kappa$, $\delta$, $p^+$ and
  $\norm{\abs{Du}^\px}_1$.
\end{proposition}
\begin{proof}
  Since $p^+<\infty$, we can choose $\delta_1$ (depending on $p^+$) so
  small such that $c_{\log}(p|_{2Q_j})\leq\delta_1$ implies~\eqref{eq:assQj}:
  $\sigma p^-_{2Q_j} \leq p^+_{Q_j}$.

  By $u-w\in W^{1,\px}_0(2Q_j)\cap W^{1,p_Q}_0(2Q_j)$ it follows
  from the equations for $u$ and $w_j$ that
  \begin{align*}
    (I) &:= \dint_{2Q_j}\big(A(\cdot,Du)-A(\cdot,Dw_j)\big)\cdot
    \big(Du-Dw_j\big) \,dx
    \\
    &=\dint_{2Q_j}\big(A_j(Dw_j)-A(\cdot,Dw_j)\big)\cdot
    \big(Du-Dw_j\big) \,dx+ \dint_{2Q_j}
    A(\cdot,G)\cdot\big(Du-Dw_j\big) \,dx
    \\
    &=: (II) + (III).
  \end{align*}
  By Young's inequality with $\gamma>0$ we have that
  \begin{align*}
    (III) &\leq c\, \gamma^{1-(p^-)'}\dint_{2Q_j}\abs{G}^{\px}
    \,dx+\gamma\dint_{2Q_j}\abs{Du}^\px+\abs{Dw_j}^\px \,dx.
  \end{align*}
  With~\eqref{eq:QjGlambda}, \eqref{eq:Qjlambda} and
  Lemma~\ref{lem:dwdu} it follows that
  \begin{align*}
    (III) &\leq \epsilon\, c\, \gamma^{1-(p^-)'}\lambda+\gamma c\,
    \lambda = \big(\epsilon\, c\, \gamma^{1-(p^-)'}+c\,\gamma\big)
    \lambda.
  \end{align*}
 The factor in front
  of~$\lambda$ is small if $\gamma$ is small and (then) $\epsilon$ is
  small. 

  It remains to estimate $(II)$. We divide the domain of integration
  in~$(II)$ into the sets
  \begin{align*}
    H_1 &:=\{ x\in {2Q_j}:\abs{Dw_j(x)}\geq1\}	,	
    \\
    H_2 &:=\{ x\in {2Q_j}:1\geq\abs{Dw_j(x)}\geq h(x)\},		
    \\
    H_3 &:=\{ x\in {2Q_j}:h(x)\geq\abs{Dw_j(x)}\}.
  \end{align*}
  We define
  \begin{align*}
    (II_k) &:= \dashint_{2Q_j} \chi_{H_k}
    \bigabs{A_j(Dw_j)-A(\cdot,Dw_j)} \, \abs{Du-Dw_j} \,dx \qquad
    \text{for $k=1,2,3$},
  \end{align*}
  then $(II) \leq (II_1) + (II_2) + (II_3)$. 

  We begin the easiest term~$(II_3)$.  By Young's inequality with
  $\gamma>0$ we estimate pointwise on $H_3$
  \begin{align*}
    \lefteqn{\bigabs{A_j(Dw_j)-A(\cdot,Dw_j)}\, \abs{Du-Dw_j}} \qquad
    &
    \\
    &\leq c\, \big( \abs{Dw_j}^{p_j-1} + \abs{Dw_j}^{\px-1}\big) (
    \abs{Du} + \abs{Dw_j})
    \\
    &\leq c\, \gamma^{1-(p^-)'} \big(\abs{Dw_j}^\px+\abs{Dw_j}^{p_j}
    \big)+ \gamma \big(\abs{Du}^{p_j}+\abs{Du}^\px\big)
    \\
    &\leq c\, \gamma^{1-(p^-)'} h(\cdot)+\gamma
    \big(\abs{Du}^{p_j}+\abs{Du}^\px\big),
  \end{align*}  
  as $h(x)\geq h(x)^\alpha$ for any $\alpha\geq 1$. This,
  \eqref{eq:QjGlambda}, Lemma~\ref{lem:upj} and~\eqref{eq:Qjlambda}
  imply 
  \begin{align*}
    (II_3) &\leq c\, \gamma^{1-(p^-)'} \dint_{2Q_j}h
    \,dx+\gamma\dint_{2Q_j}\abs{Du}^{p_j}+\abs{Du}^\px \,dx
    \\
    &\leq \big(c\, \epsilon\gamma^{1-(p^-)'} + \gamma\big) \lambda.
  \end{align*}
  Again the factor in front of~$\lambda$ is small if $\gamma$ is small
  and (then) $\epsilon$ is small.

  For the remaining terms $(II_1)$ and $(II_2)$ we have to use the
  closeness of $A_j$ to $A$ (more preciselly the smallness of $c_{\log}(p|_{2Q_j})$). In particular, by~\eqref{eq:diffAlog} and
  Young's inequality we have pointwise on~$2Q_j$
  \begin{align*}
    \begin{aligned}
      \lefteqn{\abs{(A_j(Dw_j)-A(\cdot,Dw_j)\big)\cdot\big(Du-Dw_j\big)}}
      &
      \\
      &\leq c\,\abs{\px-{p_j}} \bigabs{\log\abs{Dw_j}} \big(
      (\abs{Dw_j}^{{p_j}-1}+\abs{Dw_j}^{\px-1}\big)
      (\abs{Du}+\abs{Dw_j})
      \\
      &\leq c\,\abs{\px-{p_j}} \bigabs{\log\abs{Dw_j}} \big(
      \abs{Dw_j}^\px + \abs{Dw_j}^{p_j} + \abs{Du}^\px +
      \abs{Du}^{p_j}\big).
    \end{aligned}  
  \end{align*}

  This implies for $k=1,2$
  \begin{align*}
    (II_k) &\leq \!\dashint_{2Q_j} \!\! \chi_{H_k} \abs{p(\cdot)-p_j}
    \abs{\log \abs{Dw_j}}\big( \abs{Dw_j}^\px \!+\!
    \abs{Dw_j}^{p_j} \!+\!  \abs{Du}^\px \!+\! \abs{Du}^{p_j}\big) dx.
  \end{align*}
  After applying H{\"o}lder's estimate for the exponents
  $(2\sigma',2\sigma',\sigma)$ and using Lemma~\ref{lem:dwdu},
  Lemma~\ref{lem:wpj}, \eqref{eq:Qjlambda} and Lemma~\ref{lem:upj} we
  get for $k=1,2$
  \begin{align*}
    (II_k) &\leq \bigg( \dashint_{2Q_j} \abs{p(\cdot)-p_j}^{2\sigma'}
    \bigg)^{\frac 1{2\sigma'}} \bigg( \dashint_{2Q_j} \chi_{H_k}
    \abs{\log \abs{Dw_j}}^{2\sigma'}\,dx \bigg)^{\frac
      1{2\sigma'}} \lambda
    \\
    &=: (IV)\, (V_k)\, \lambda.
  \end{align*}
  Due to Lemma~\ref{lem:ppjavg} we have
  \begin{align*}
    (IV) &\leq \frac{c\, (p^+)^2 c_{\log}(p)}{\log (e+\max
      \set{1/\ell(Q_j), \ell(Q_j), \abs{\cent(Q_j)}})}.
  \end{align*}
  Since $\abs{Dw_j} \geq 1$ on $H_1$, we have
  \begin{align*}
    (V_1) &\leq \bigg( \dashint_{2Q_j} \chi_{H_1} \big(\log
    (\abs{Dw_j}^{p_j})\big)^{2\sigma'}\,dx \bigg)^{\frac 1{2\sigma'}}
    \\
    &\leq \bigg( \dashint_{2Q_j} \chi_{H_1} \big(\log
    (e+\abs{Dw_j}^{p_j})\big)^{2\sigma'}\,dx \bigg)^{\frac
      1{2\sigma'}}.
  \end{align*}
  Using the estimate $\log(e+t) \leq \log(e+t/\lambda) + \log(e+\lambda)$
  we get
  \begin{align*}
    (V_1) &\leq \bigg( \dashint_{2Q_j} \chi_{H_1} \big(\log
    (e+\frac{\abs{Dw_j}^{p_j}}{\lambda})\big)^{2\sigma'}\,dx
    \bigg)^{\frac 1{2\sigma'}} + c\,\log(e+\lambda).
  \end{align*}
  Due to Lemma~\ref{lem:wpj} we have
  \begin{align*}
    \dashint_{2Q_j} \abs{Dw_j}^\px \,dx &\leq c\, \lambda.
  \end{align*}
Lemma~\ref{lem:lnfmean} and the previous estimate imply
  \begin{align*}
    (V_1) &\leq c + c\,\log(e+\lambda) \leq c\, \log(e+\lambda).
  \end{align*}  
  From $\norm{\abs{Du}^\px}_1 \leq c$, we know that $\int_{2\Omega}
  \abs{Du}^\px \leq c$. This and~\eqref{eq:Qjlambda} imply
  \begin{align*}
    \lambda \leq \frac{c}{\abs{Q_j}}.
  \end{align*}
 Therefore we gain
  \begin{align*}
    (V_1) &\leq c\, \log(e + 1/\ell(Q_j)).
  \end{align*}
  Since $(e+ \abs{x})^{-m} = h(x) \leq \abs{Dw_j(x)} \leq 1$ on $H_2$,
  we have
  \begin{align*}
    (V_2) &\leq \bigg( \dashint_{2Q_j} \big(\log
    (e+\abs{x})\big)^{2\sigma'}\,dx \bigg)^{\frac 1{2\sigma'}}
    \\
    &\leq c\,\log
    (e+\max \set{\ell(Q_j), \abs{\cent(Q_j)}}). 
  \end{align*}
  Overall, we have
  \begin{align*}
    (V_1) + (V_2) &\leq c\,\log (e+\max \set{1/\ell(Q_j), \ell(Q_j),
      \abs{\cent(Q_j)}}).
  \end{align*}
  The estimates for $(IV)$, $(V_1)$ and $(V_2)$ imply
  \begin{align*}
    (II_1) + (II_2) &\leq c\,(p^+)^2\,c_{\log}(p)\,\lambda.
  \end{align*}
  The factor in front of~$\lambda$ is small if $\delta_1$ is small
  (for fixed upper bound of $p^+$).

  Combining the estimates for $(III)$, $(II_1)$, $(II_2)$ and $(II_3)$
  proves the claim.
\end{proof}
%

\begin{remark}
\label{rem:pVMO}
The condition on $p$ can be weakened for the last estimate. Indeed, we can replace the smallness of the log H{\"o}lder constant by the assumption that the oscillations of $p$ are small:
\begin{align*}
 \bigg(\dashint_{2Q_j} \abs{p(\cdot)-p_j}^{2\sigma'}dx\bigg)^\frac1{2\sigma'}
&\leq \frac{\delta_1}{\log(e+\max \set{1/\ell(Q_j), \ell(Q_j),
      \abs{\cent(Q_j)}})}.
\end{align*}
Or as a counterpart for the pointwise vanishing condition, the following $\setVMO$ condition:
\begin{align*}
 \frac{ \bigg(\dashint_{2Q_j} \abs{p(\cdot)-p_j}^{2\sigma'}dx\bigg)^\frac1{2\sigma'}}{\log(e+\max \set{1/\ell(Q_j), \ell(Q_j),
      \abs{\cent(Q_j)}})}\to 0,
\end{align*}
when $\ell(Q_j)\to 0$ or $\cent(Q_j)\to\infty$. This is a first step to weaken the pointwise vanishing log H\"older continuity by an integral vanishing oscillation condition on the exponent.

However, up to know we still require the (not small) log H{\"o}lder continuity to be able to apply the key estimate (Lemma~\ref{lem:pjensen}). 
\end{remark}

\subsection{Main results}
\label{ssub:main}

Let us present the main results of this paper on local higher
integrability.
\begin{theorem}
  \label{thm:local}
  Let $\Omega \subset \Rn$ be a cube and let $u$ be a solution
  of~\eqref{eq:difeq} on $2\Omega$. Further, let $p\in\PPln(2\Omega)$,
  $1<p^-\leq p^+<\infty$, $q \geq 1$, and $\abs{G}^\px \in
  L^q(2\Omega)$. Then there exists a $\delta_1>0$ such that
  $c_{\log}(p|_{2\Omega})\leq \delta_1$ implies
  \begin{align*}
    \bigg(\dint_{\Omega}\abs{Du}^{\px q} \,dx\bigg)^\frac{1}{q}\leq c
    \dint_{2\Omega}\abs{Du}^{\px} \,dx+
    c\bigg(\dint_{2\Omega}(\abs{G}^\px+h)^q \,dx\bigg)^\frac{1}{q}.
  \end{align*}
  Here $\delta_1$ and $c$ only depend on
  $\norm{\abs{Du}^\px}_{1,2\Omega}$, $p^-,p^+$, and $q$. The function $h(x):=(e+\abs{x})^{-2n}$. 
\end{theorem}
\begin{theorem}
  \label{thm:localv}
  Let $\Omega \subset \Rn$ be a cube and let $u$ be a solution
  of~\eqref{eq:difeq} on $2\Omega$. Further, let $p\in\PPvln(2\Omega)$,
  $1<p^-\leq p^+<\infty$, $q \geq 1$, and $\abs{G}^\px \in
  L^q(2\Omega)$. Then
  \begin{align*}
    \bigg(\dint_{\Omega}\abs{Du}^{\px q} \,dx\bigg)^\frac{1}{q}\leq c
    \dint_{2\Omega}\abs{Du}^{\px} \,dx+
    c\bigg(\dint_{2\Omega}(\abs{G}^\px+h)^q \,dx\bigg)^\frac{1}{q},
  \end{align*}
  where $c$ depends on $\norm{\abs{Du}^\px}_{1,2\Omega}$, $p^-,p^+$,
  $q$ and $p$ via the vanishing $\log$-H{\"o}lder continuity. The function $h(x):=(e+\abs{x})^{-2n}$. 
\end{theorem}
We postpone the proof of theses theorems until the end of this
subsection.
The above theorems on local higher integrability have a global
counterpart.
\begin{corollary}
  \label{cor:global}
  Let $u$ be a solution of~\eqref{eq:difeq} on~$\Rn$ with $Du
  \in L^\px(\Rn)$, let $p \in \PPln(\Rn)$, $q \geq 1$, and $G^\px \in
  L^1(\Rn) \cap L^q(\Rn)$. Then there exists $\delta_1>0$ such that
  $c_{\log}(p)\leq \delta_1$ implies
  \begin{align*}
    \norm{\abs{Du}^\px}_q\leq c\norm{\abs{G}^\px+h}_q.
  \end{align*}
  Here $\delta_1$ and $c$ depend on $\norm{\abs{Du}^\px}_1$
  and~$q$. The function $h(x):=(e+\abs{x})^{-2n}$. 
  %
\end{corollary}
\begin{corollary}
  \label{cor:globalv}
  Let $u$ be a solution of~\eqref{eq:difeq} on~$\Rn$ with $Du
  \in L^\px(\Rn)$, let $p \in \PPvln(\Rn)$, $q \geq 1$, and $G^\px \in
  L^1(\Rn) \cap L^q(\Rn)$. Then
  \begin{align*}
    \norm{\abs{Du}^\px}_q\leq c\norm{\abs{G}^\px+h}_q,
  \end{align*}
  where $c$ depends on $\norm{\abs{Du}^\px}_{1}$, $p^-,p^+$,
  $q$ and $p$ via the vanishing $\log$-H{\"o}lder continuity. The function $h(x):=(e+\abs{x})^{-2n}$. 
\end{corollary}
The two corollaries are immediate consequences of the two theorems
above. Just apply Theorem~\ref{thm:local} and
Theorem~\ref{thm:localv}, resp., to $\Omega = (-R,R)^n$, multiply by
$\abs{\Omega}$ and let $R \to \infty$.

Before we proof Theorem~\ref{thm:local} and Theorem~\ref{thm:localv}
we need a few auxiliary results. First, we need the interior
regularity of $p_j$-harmonic functions, which was proven by Uhlenbeck for systems and $p\geq 2$ in \cite[(3.2)Theorem]{Uhl77}. By duality the estimate holds also in the case $p\leq2$; see also~\cite{DieSV09} for a more general version of the estimate below. 
\begin{theorem}
  \label{thm:hp1}
  The $p_j$-harmonic function $w_j$ satisfies
  \begin{align*}
    \sup_{\frac{3}{2}Q_j}\abs{Dw_j} \leq
    c\,\left(\dint_{2Q_j}\abs{Dw_j}^{p_j} \,dx\right)^\frac{1}{p_j}.
  \end{align*}
\end{theorem}

This implies 
\begin{lemma}
  \label{lem:dw}
  With the same assumptions as in Proposition~\ref{pro:comparison} on $Q_j$ we have
  \begin{align*}
    \bigg(\dint_{2Q'} \abs{Dw_j}^{\sigma p_j} \,dx\bigg)^{\frac
      1\sigma}&\leq c\,\lambda.
    \\
    \dint_{2Q'} \abs{Dw_j}^\px \,dx&\leq c\, \lambda + c\,
    \dashint_{2Q'} (e+\abs{x})^{-2n} \,dx
  \end{align*}
  for all $Q'\in \Delta_{\Omega}$ with $Q' \subsetneqq Q_i$ and $Q'
  \cap \Uel \not= \emptyset$.
\end{lemma}
\begin{proof}
  The first estimate follows directly from Theorem~\ref{thm:hp1} and
  Lemma~\ref{lem:wpj}. 

  By the key estimate (Lemma~\ref{lem:pjensen}) applied to
  the exponent $\sigma \frac{p_j}{\px}\geq 1$ at the point $y_j$ and
  Remark~\ref{rem:dwK} we have
  \begin{align*}
    \bigg( \dashint_{2Q'} \abs{Dw_j}^\px
    \,dx\bigg)^{\sigma\frac{p_j}{p(y_j)}}&\leq c\, \dint_{2Q'}
    \abs{Dw_j}^{\sigma p_j} + c\,\dashint_{2Q'} (e+\abs{y})^{-\sigma
      2n}\,dy.
  \end{align*}
  This and the first part of the lemma imply
  \begin{align*}
    \dashint_{2Q'} \abs{Dw_j}^\px \,dx&\leq c\, \lambda + c\,
    \dashint_{2Q'} (e+\abs{y})^{-2n}\,dy.
  \end{align*}
  Since $Q' \cap \Uel \not= \emptyset$, the last integral is bounded
  by $c\, \epsilon\lambda$ due to~\eqref{eq:QjGlambda}.
\end{proof}
We are now prepared to show our redistributional
estimate~\eqref{eq:glb}. 
\begin{proposition}
  \label{prop:goodl}
  There exists $\kappa \geq 2^n$ such that for every $\delta>0$ we
  find $\epsilon>0$ and $\delta_1>0$ with the following property for all $\lambda \geq
  \lambda_0$: If
  $c_{\log}(p|_{2Q_j})\leq \delta_1$ for all $Q_j\cap \Uel\neq \emptyset$, then there holds
  \begin{align*}
    \abs{\Uel}\leq \delta\,
    \abs{\mathcal{O}_\lambda}.
  \end{align*}
  The value of $\epsilon$ and $\delta_1$ depends on $\delta$, $p^+$,
  and $\norm{\abs{Du}^\px}_{L^1(2\Omega)}$.
\end{proposition}
\begin{proof}
  We will choose the exact value of $\epsilon$, $\kappa$ and
  $\delta_1$ during the proof.  Let $\lambda \geq \lambda_0$.  We
  already know (see~\eqref{eq:glb2}) that it suffices to prove
  \begin{align}
    \label{eq:glb3}
      \abs{Q_j \cap \Uel} \leq \delta\, \abs{Q_j} \qquad
      \text{ for all $j \in \setN$ with $Q_j \cap \Uel
        \not=\emptyset$.}
  \end{align}
  So let us assume in the following that $Q_j \cap \Uel
  \not= \emptyset$.

  Let $y \in Q_j \cap \Uel$. By definition
  of~$\Uel$, see~\eqref{eq:OlUl} we have
  \begin{align*}
    \Ms(\abs{Du}^\px)(y) &>\kappa \lambda,
    \\
    \Mss(\gh)(y) &\leq\epsilon\lambda.
  \end{align*}
  Therefore, there exists a cube $\frC_y \in \Delta_{2\Omega}$ with
  \begin{align*}
    \dashint_{2\frC_y} \abs{Du}^\px \,dx > \kappa \lambda.
  \end{align*}
  First we assume that $\kappa\geq 2^n$.  Then obviously $\kappa
  \lambda > \lambda_0$ and hence by definition of~$\lambda_0$ the cube
  $\frC_y$ must be a strict dyadic sub-cube of~$\Omega$.  Hence the
  predecessor $\frC_y^\pre$ satisfies $\frC_y^\pre \in
  \Delta_{\Omega}$ and
  \begin{align*}
    \dashint_{2\frC_y^\pre} \abs{Du}^\px\,dx \geq 2^{-n}
    \dashint_{2\frC_y} \abs{Du}^\px \,dx > 2^{-n} \kappa \lambda \geq 
    \lambda.
  \end{align*}
  Since $Q_j$ was chosen to be the maximal cube of $\Delta_{\Omega}$
  containing~$y$ with this property, it follows that $\frC_y^\pre
  \subset Q_j$.

  By~\eqref{eq:diffAV} and Lemma~\ref{lem:dw} (using $\frC_y
  \subsetneqq Q_j$, since $\frC_y^\pre \subset Q_j$) we have
  \begin{align*}
    \kappa\lambda &< \dashint_{2\frC_y} \abs{Du}^\px\,dx
    \\
    &\leq c_4 \dashint_{2\frC_y} \abs{Dw_j}^\px\,dx + c_4
    \dashint_{2\frC_y} \big(A(\cdot,Du)- A(\cdot,Dw_j))\cdot (Du -
    Dw_j)\,dx
    \\
    &\leq c\, \lambda + c_4 \dashint_{2\frC_y} \big(A(\cdot,Du)-
    A(\cdot,Dw_j))\cdot (Du - Dw_j)\,dx.
  \end{align*}
  The constant~$c$ depends on $p^+$ and
  $\norm{\abs{Du}^\px}_{L^1(2\Omega)}$.  So for $\kappa$ large (which
  finally fixes~$\kappa$) we can absorb $c\,\lambda$ into
  $\kappa\lambda$. By multiplication with $\abs{\frC_y}$ we get
  \begin{align*}
    \kappa\lambda \abs{\frC_y} &\leq c\, \int_{2\frC_y}
    \big(A(\cdot,Du)- A(\cdot,Dw_j))\cdot (Du - Dw_j)\,dx.
  \end{align*}
  The collection of $\frC_y$ covers the set $Q_j \cap \Uel$. Since all
  these cubes are $\Omega$-dyadic, there exists a sub-family
  $\set{\frC_{j,k}}_k$ of maximal $\Omega$-dyadic cubes. We sum the
  previous inequality over these cubes to get
  \begin{align*}
    \kappa\lambda \abs{Q_j \cap \Uel} &\leq \kappa\lambda \sum_k
    \abs{\frC_{j,k}}
    \\
    &\leq c\, \sum_k \int_{2\frC_{j,k}} \big(A(\cdot, Du)-
    A(\cdot,Dw_j))\cdot (Du - Dw_j)\,dx
    \\
    &\leq c\, \int_{2Q_j} \big(A(\cdot, Du)- A(\cdot,Dw_j))\cdot (Du -
    Dw_j)\,dx.
  \end{align*}
  Due to Proposition~\ref{pro:comparison} we can find for every
  $\delta>0$ a proper choice of~$\delta_1>0$ and $\epsilon>0$ such
  that the last integral is bounded by $\abs{2Q_j} \delta\,\lambda$. Hence,
  \begin{align*}
    \kappa\lambda \abs{Q_j \cap \Uel} &\leq \abs{2Q_j} \delta
    \,\lambda.
  \end{align*}
  In other words (using $\kappa \geq 2^n$)
  \begin{align*}
    \abs{Q_j \cap \Uel} &\leq \delta\, \kappa^{-1} 2^n \abs{Q_j} \leq
    \delta\, \abs{Q_j}.
  \end{align*}
  This proves the claim.
\end{proof}

We can now proof our first main result on local higher integrability.
\begin{proof}[Proof of Theorem~\ref{thm:local}] 
  If $q \in [1,m_0]$, then the claim follows directly from
  Corollary~\ref{cor:gehring}. So we can assume in the following $q
  > m_0 > 1$.  Fix $\kappa$ as in Proposition~\ref{prop:goodl}.
  Then for $\delta:= \frac 12 \kappa^{-q}$ let $\epsilon$ and
  $\delta_1$ be chosen (depending on~$\delta$) as in
  Proposition~\ref{prop:goodl} such that $\abs{\Uel} \leq \delta
  \abs{\mathcal{O}_\lambda}$ for every $\lambda \geq \lambda_0$ with
  $\lambda_0 = \dashint_{2\Omega} \abs{Du}^\px \,dx$.

  We estimate
  \begin{align}
    \label{eq:5}
    \begin{aligned}
      \dint_{\Omega}\abs{Du}^{\px q} \,dx=&
      \frac{1}{\abs{\Omega}}\bigg(\int_{0}^{\kappa\lambda_0}
      q\,\lambda^{q-1} \bigabs{\Omega \cap \bigset{\abs{Du}^\px >
          \lambda}}\,d\lambda
      \\
      &\hphantom{\frac{1}{\abs{\Omega}}\bigg(}
      +\int_{\kappa\lambda_0}^\infty q\,\lambda^{q-1} \bigabs{\Omega
        \cap \bigset{\abs{Du}^\px > \lambda}} \,d\lambda\bigg)
      \\
      &\leq \kappa^q \lambda_0^q + \frac{1}{\abs{\Omega}} \int_{\kappa
        \lambda_0}^\infty q\,\lambda^{q-1} \bigabs{
        \bigset{\Ms(\abs{Du}^\px) > \lambda}}\, d\lambda
      \\
      &=: \kappa^q \lambda_0^q + (I).
    \end{aligned}
  \end{align}
  For $\Lambda \geq \kappa \lambda_0$ define
  \begin{align*}
    (I_\Lambda) &= \int_{\kappa \lambda_0}^\Lambda q\,\lambda^{q-1}
    \bigabs{\bigset{\Ms(\abs{Du}^\px)>\lambda}} \,d\lambda,
  \end{align*}
  then $(I) = \lim_{\Lambda \to \infty} (I_\Lambda)$. Substitution
  gives
  \begin{align*}
    (I_\Lambda) &:= \kappa^q \int_{\lambda_0}^{\Lambda/\kappa}
    q\,\lambda^{q-1} \bigabs{\bigset{\Ms(\abs{Du}^\px)>\kappa
        \lambda}} \,d\lambda.
  \end{align*}
  From $\abs{\Uel} \leq \delta \abs{\mathcal{O}_\lambda}$ for $\lambda
  \geq \lambda_0$ it follows that
  \begin{align*}
    \bigabs{\bigset{\Ms(\abs{Du}^\px)>\kappa \lambda}} &\leq \;
    \bigabs{\bigset{\Mss(\gh)>\epsilon \lambda}}
    \\
    &\quad + \delta \bigabs{\bigset{\Ms(\abs{Du}^\px)>\lambda}}.
  \end{align*}
  Hence,
  \begin{align*}
    (I_\Lambda) &\leq \kappa^q \int_{\lambda_0}^{\Lambda/\kappa}
    q\,\lambda^{q-1} \bigabs{\bigset{\Mss(\gh)>\epsilon \lambda}}
    \,d\lambda
    \\
    &\quad + \kappa^q \delta \int_{\kappa\lambda_0}^{\Lambda/\kappa}
    q\,\lambda^{q-1} \bigabs{\bigset{\Ms(\abs{Du}^\px)>\lambda}}
    \,d\lambda  
 \\
    &\quad + \kappa^q \delta \int_{\lambda_0}^{\kappa\lambda_0}
    q\,\lambda^{q-1} \bigabs{\bigset{\Ms(\abs{Du}^\px)>\lambda}}
    \,d\lambda.
  \end{align*}
  Since $\kappa^q \delta = \frac 12$, the second term is bounded by
  $\frac 12 (I_\Lambda)$. The last term can be estimated as in \eqref{eq:5}. This implies 
  \begin{align*}
    (I_\Lambda) &\leq 2\,\kappa^q \int_{\lambda_0}^{\Lambda/\kappa}
    q\,\lambda^{q-1} \bigabs{\bigset{\Mss(\gh)>\epsilon \lambda}}
    \,d\lambda+\abs{\Omega}\lambda_0^q
    \\
    &= 2\,\kappa^q \epsilon^{-q}
    \int_{\epsilon\lambda_0}^{\epsilon\Lambda/\kappa} q\,\lambda^{q-1}
    \bigabs{\bigset{\Mss(\gh)>\lambda}} \,d\lambda+\abs{\Omega}\lambda_0^q
    \\
    &\leq 2\,\kappa^q \epsilon^{-q} \int_{\Omega}
    \big(\Mss(\gh)\big)^q\,dx+\abs{\Omega}\lambda_0^q.
  \end{align*}
  The boundedness of the operator $\Ms$ on $L^q(2\Omega)$ (using
  $q> m_0$) implies
  \begin{align*}
    (I_\Lambda) &\leq c\,\kappa^q \epsilon^{-q} \int_{2\Omega}
    \gh^q\,dx+\abs{\Omega}\lambda_0^q.
  \end{align*}
  The constant depends on~$q$, but the lower bound $q > m_0$
  ensures that the operator norm of $\Ms$ is uniformly bounded.  We
  pass to the limit $\Lambda \to \infty$, combine this
  with~\eqref{eq:5} and use the definition of~$\lambda_0$ to get
  \begin{align*}
    \dint_{\Omega}\abs{Du}^{\px q} \,dx &\leq c\,\epsilon^{-q}
    \int_{2\Omega} \gh^q\,dx + \bigg(2\kappa \dashint_{2\Omega} \abs{Du}^\px\bigg)^q
    \,dx.
  \end{align*}
  This proves the claim.
\end{proof}
\begin{proof}[Proof of Theorem~\ref{thm:localv}]
%
%
 Let $q\geq m_0$ and choose $\delta_1>0$ as in
  Theorem~\ref{thm:local}. Since $c_{\log}(p|_{2\Omega})$ does not need to be
  smaller that $\delta_1$, we cannot apply Theorem~\ref{thm:local}
  directly. 
Let $\set{Q_j}_{j\in\setN}$ be the \Calderon Zygmund covering introduced at the beginning of this section. We will show the following:

For every $\delta_1>0$ there exists an $\epsilon_0$, such that for every $\epsilon\leq\epsilon_0$ and all $Q_j$ with $Q_j\cap\Uel\neq \emptyset$ we have 
$c_{\log}(p|_{2Q_j})\leq \delta_1$. Then Proposition \ref{prop:goodl} can be applied and the result follows as in Theorem \ref{thm:local}.

  By the vanishing $\log$-H{\"o}lder continuity, we find $r,R>0$ such
  that
  \begin{alignat*}{2}
    \abs{p(x)-p(y)} &\leq \frac{\delta_1}{\log
      (e+1/\abs{x-y})}
  \end{alignat*}
  for all $x,y$ with $\abs{x-y} \leq r$ and all $x,y \in
  {2\Omega}\setminus Q_R(0)$. We can choose $R$ large enough such that additionally
  \begin{align*}
    \abs{p(z) - p_\infty} &\leq \frac{\delta_1}{\log(e +
      \abs{z})}
  \end{align*}
  for all $z \in {2\Omega} \setminus Q_R(0)$. 
Therefore, if $2Q_j \subset {2\Omega} \setminus Q_R(0)$, then
  $c_{\log}(p|_{2Q_j}) \leq \delta_1$. On the other hand if the length of $Q_j$ is smaller than~$r$, then
  $c_{\log}(p|_{2Q_j}) \leq \delta_1$.

It leaves the case when $\abs{Q_j}\geq r^n$ and $2 Q_j\cap Q_R(0)\neq\emptyset$.
If now $Q_j\cap\Uel\neq\emptyset$, then there exists a $c$ depending on $r$ and $R$, (but independent of $Q_j$), such that
\begin{align*}
 \lambda \abs{2 Q_j}\leq \int_{2\Omega}\abs{Du}^\px dx\leq c\int_{2Q_j}hdx\leq c\epsilon\lambda\abs{2Q_j}.
\end{align*}
This is never the case, whenever $\epsilon$ is small enough. Therefore the proof is complete.
\end{proof}

\end{document}